\documentclass[]{amsart}

\usepackage{amsfonts}
\usepackage{amscd}
\usepackage{epsfig}
\usepackage{epic,eepic}
\usepackage{graphicx}
\usepackage{xypic}
\usepackage{psfrag}
\usepackage{amsmath}
\usepackage{amsthm}
\usepackage{setspace}
\usepackage{url}

\title{On the Dimension of the Hilbert Scheme of Curves}
\author{Dawei Chen}
\date{}

\theoremstyle{plain}
\newtheorem{theorem}{Theorem}[section]
\newtheorem{lemma}[theorem]{Lemma}
\newtheorem{proposition}[theorem]{Proposition}

\newtheorem{conjecture}[theorem]{Conjecture}

\theoremstyle{definition}
\newtheorem{remark}[theorem]{Remark}

\begin{document}
\bibliographystyle{plain}
\maketitle

\begin{abstract}
Consider a component of the Hilbert scheme whose general point
corresponds to a degree $d$ genus $g$ smooth irreducible and
nondegenerate curve in a projective variety $X$. We give lower
bounds for the dimension of such a component when $X$ is $\mathbb P^3, \ 
\mathbb P^4$ or a smooth quadric threefold in $\mathbb P^4$
respectively. Those bounds make sense from the asymptotic viewpoint
if we fix $d$ and let $g$ vary. Some examples are constructed using
determinantal varieties to show the sharpness of the bounds for
$d$ and $g$ in a certain range. The results can also be applied to
study rigid curves.
\end{abstract}

\tableofcontents

\section{Introduction}
In this section, we briefly recall some basic facts about Hilbert
schemes, and state the main results of this paper.

Let $P$ be the Hilbert polynomial of a subscheme in $\mathbb P^{r}$.
We can ask if there exists a good parameter space $\mathcal {H}_{P,
r}$ parametrizing all the subschemes that have $P$ as their Hilbert
polynomial. Grothendieck proved the following
fundamental result on the existence of $\mathcal {H}_{P, r}$.

\begin{theorem}
There exists a fine moduli space $\mathcal {H}_{P, r}$. Moreover, it
is a projective scheme.
\end{theorem}

Very few facts about the global properties of $\mathcal H_{P, r}$
have been obtained. However, the connectedness of $\mathcal H_{P,
r}$ has been proved in Hartshorne's thesis.

\begin{theorem}
The Hilbert scheme $\mathcal H_{P,r}$ is connected for any $P$ and
$r$.
\end{theorem}

Here curves are our main interests. The Hilbert polynomial $P$ of a
curve is a linear function with leading coefficient $d$ and constant
term $1-g$, where $d$ and $g$ are the degree and genus of the
curve. In this case, we use the notation $\mathcal H_{d,g,r}$ in
stead of $\mathcal H_{P,r}$. Sometimes we will also simply use
$\mathcal H$ when there is no confusion.

Consider the dimension of $\mathcal H$. We have the following
result. See, for instance \cite[Section 1.E]{HM}, for related references.

\begin{theorem}
\label{dim} Let $C$ be a 1-dimensional subscheme in $\mathbb
P^{r}$ such that $[C]\in\mathcal H_{d,g,r}$. The tangent space of $\mathcal H$ at $[C]$ can be identified as $$ T_{[C]}\mathcal H = H^{0}(C,\mathcal N_{C/\mathbb P^{r}}), $$
where $\mathcal N_{C/\mathbb P^{r}}$ is the normal sheaf of $C$ in
$\mathbb P^{r}$. Moreover, if $C$ is a local complete intersection,
then $$ h^{0}(C, \mathcal N_{C/\mathbb P^{r}})-h^{1}(C, \mathcal
N_{C/\mathbb P^{r}})\leq \mbox{dim}_{[C]}\mathcal H\leq h^{0}(C,
\mathcal N_{C/\mathbb P^{r}}). $$
\end{theorem}

Let $U$ be a component of $\mathcal H_{d,g,r}$ whose general point
corresponds to a smooth irreducible and nondegenerate curve $C$. Also let
$l_{d,g,r}$ be the lower bound for the dimension of all such components $U$.
Our aim is to estimate $l_{d,g,r}$. Define a number $h_{d,g,r} = \chi(\mathcal N_{C/\mathbb P^{r}}) = h^{0}(C, \mathcal N_{C/\mathbb P^{r}})-h^{1}(C, \mathcal N_{C/\mathbb
P^{r}}) = (r+1)d-(r-3)(g-1)$. By Theorem \ref{dim}, we know that
$l_{d,g,r}\geq h_{d,g,r}$. However, this bound $h_{d,g,r}$ may not be good in many cases.

For the beginning case $r = 3$, $h_{d,g,3} = 4d$ is independent of $g$. If we fix $d$ and let $g$ vary, the genus of a degree
$d$ irreducible and nondegenerate curve in $\mathbb P^{3}$ can be as large as the Castelnuovo bound $\pi(d,3) = \frac{d^{2}}{4}+O(d)$. One can refer to \cite[Section 3]{H} for a good introduction on the Castelnuovo theory and related results. When $g$ approaches
$\pi(d,3)$, we can compute $l_{d,g,3}$ explicitly. Roughly speaking,
$l_{d,g,3}$ is asymptotically equal to $g$, which is much larger than $4d$. Therefore, it would be nice if we can come up with an improved lower bound.

\begin{theorem}
\label{r=3}
Define an integer-valued function $\mu(d,g)$ in the range $g^{2}\geq d^{3}$ as follows,
$$ \mu(d,g) = 1+\lfloor\frac{d^{2}-3d-2g}{g+d+\sqrt{g^{2}-d^{3}+4dg+4d^{2}}}\rfloor, $$
where $\lfloor \cdot\rfloor$ is the floor function.
Then for any $d \geq 3$ and $g \leq \pi(d,3)$, we have
$$ l_{d,g,3} \geq \begin{cases}
4d, &\text{if $g^{2}< d^{3}$}; \\
4d+g-1-\mu(d,g)d, &\text{if $g^{2}\geq d^{3}$}.
\end{cases} $$
\end{theorem}

The function $\mu$ invloved in Theorem \ref{r=3} may look confusing. But let us analyze
this new bound a little bit. If $g^{2}\geq d^{3}$, we always have $g-1-\mu(d,g)d>0$. Moreover, if we fix $d$, $g-1-\mu(d,g)d$ is an
increasing function of $g$. It implies that in the range $g^{2}\geq d^{3}$, $l_{d,g,3}$ is strictly larger than the expected dimension $4d$. It actually goes up to $g$ when the genus approaches the Castelnuovo bound $\pi(d,3)$, which has been already predicted by the Castelnuovo theory.

We can also present an example to show the power of this bound. Suppose $d = 100$ and $g$ can vary from $0$ to
the Castelnuovo bound $\pi(100,3) = 2401$. Pick $g = 1100$, which is large but not close to the Castelnuovo bound.
The bound $4d$ only tells us that $l_{100,1100,3}\geq 400$. However, by Theorem \ref{r=3}, we get
$l_{100, 1100, 3}\geq 1099$, which is much better.

Now consider the case $r\geq 4$. The number $h_{d,g,r} = (r+1)d-(r-3)(g-1)$ could be negative if $g$ is larger than $d$. So it makes sense to find at least a positive bound for $l_{d,g,r}$. Furthermore, it may help answer a question about rigid curves.

A rigid curve in $\mathbb P^{r}$ is a smooth irreducible and nondegenerate curve that does not have any deformation except those induced from
the automorphisms of $\mathbb P^{r}$. Apparently, rational normal curves are rigid. To the author's best knowledge, people have not found any other rigid curves. In \cite{HM}, Harris and Morrison conjectured that there does not exist a rigid curve except rational normal curves. One way to attack this conjecture
is to bound $l_{d,g,r}$. For instance, if the equality $l_{d,g,r} > $ dim PGL($r)=r^{2}+2r$ holds, there cannot exist a degree $d$ genus $g$ rigid curve in $\mathbb P^{r}$. In fact, this is one of our motivations to study $l_{d,g,r}$.

For the case $r=4$, we have the following result.
\begin{theorem}
\label{r=4}
Let $C$ be a degree $d$ genus $g$ smooth irreducible and nondegenerate curve in $\mathbb P^{4}$. Fix $d$ and let $g$ vary. If $g>3d\sqrt{d}+O(d)$, then $C$ is not rigid.
\end{theorem}

Here we could be more precise on the range of $d$ and $g$ as we have done in Theorem
\ref{r=3}. However, we choose to only focus on the asymptotic behavior, since
the order $d\sqrt{d}$ seems to be important. Currently we have not been able to extend the result to $r\geq 5$. But combining the results in \cite{CCG}, we expect the following conjecture to hold in general.

\begin{conjecture}
\label{r>4}
For $r\geq 5$, there always exists a constant $\lambda_{r}$ such that if $g\geq \lambda_{r}d\sqrt{d}+O(d)$, a degree $d$ genus $g$ smooth irreducible
and nondegenerate curve in $\mathbb P^{r}$ is not rigid.
\end{conjecture}

In addition to projective spaces, we can also study the deformation of curves on a hypersurface. The beginning case would be a smooth quadric threefold in $\mathbb P^{4}$. Since all the smooth quadrics in $\mathbb P^{4}$ are isomorphic, we fix one and denote it by $Q$. Let $\mathcal H_{d, g}(Q)$ be the union of components of the
Hilbert scheme whose general point parameterizes a degree $d$ genus $g$ smooth irreducible and nondegenerate curve on $Q$. Here a nondegenerate curve means that it is not contained in a $\mathbb P^{3}$. For a curve [$C$]$\in \mathcal H_{d,g}(Q)$, as in Theorem \ref{dim},
$\mathcal X(\mathcal N_{C/Q})= h^{0}(\mathcal N_{C/Q})-h^{1}(\mathcal N_{C/Q}) = 3d$ provides a lower bound for the dimension of any component in $\mathcal H_{d,g}(Q)$. We can still ask how good this lower bound would be. A similar result as Theorem \ref{r=4} can be established as follows.

\begin{theorem}
\label{quadric}
If $g> \frac{1}{\sqrt{2}} d\sqrt{d} + O(d)$, then the dimension of any component of $\mathcal H_{d,g}(Q)$ is strictly greater than the expected dimension $3d$. On the other hand, if $g<\frac{2}{15\sqrt{5}} d\sqrt{d} + O(d)$, then there always exists a component of $\mathcal H_{d,g}(Q)$ whose dimension equals $3d$.
\end{theorem}

Again, we only focus on the asymptotic behavior. The coefficients of $d\sqrt{d}$ might be improved by refining our techniques, but it seems hard to obtain a better order than $d\sqrt{d}$. 

Throughout the paper, we work over the complex number field. A degree $d$
genus $g$ curve in a projective space means a 1-dimensional subscheme that has $dm + 1-g$ as its Hilbert polynomial. Most of the time we will only consider smooth irreducible and nondegenerate curves.

{\bf Acknowledgements.} I am grateful to Professor Joe Harris, who first told me about this question and made many useful suggestions during the preparation of this work.

\section{The Hilbert scheme of curves in $\mathbb P^{3}$}
In this section, we will verify Theorem \ref{r=3}. Let us briefly describe the outline of the proof. Fix $d$ and let $g$ vary. On one hand, we construct some components of the Hilbert scheme with the expected dimension $4d$ when $g$ is relatively small. On the other hand, if $g$ is quite large, the curve must lie on a low degree surface. We can estimate the dimension of the deformation of the curve on that surface, which would provide a better bound than $4d$.

\subsection{Determinantal curves in $\mathbb P^{3}$}

As mentioned before, we want to construct some components of the Hilbert scheme
that have $4d$ as their dimension.

For a curve $C$ in $\mathbb P^{3}$, let
$\mathcal I_{C} = \mathcal I_{C/\mathbb P^{3}}$ denote the ideal sheaf
of $C$, and let $\mathcal N_{C}$ be the normal sheaf $\mathcal
N_{C/\mathbb P^{3}}$. Firstly, let us look at an example constructed in \cite{E}.

Consider a curve $C$ whose ideal sheaf has resolution as follows,
\begin{equation}
\label{d3}
0\rightarrow \mathcal O^{\oplus s}_{\mathbb P^{3}}(-s-1)\rightarrow \mathcal O^{\oplus (s+1)}_{\mathbb P^{3}}(-s)\rightarrow \mathcal I_{C} \rightarrow 0.
\end{equation}

It is easy to derive the determinantal model for such a curve from this resolution. Pick an $s \times (s+1)$ matrix $A$ whose entries
are general linear forms. Then the ideal sheaf of the curve defined by the determinants of all the $s\times s$ minors of $A$ has the above resolution.
Tensor the exact sequence (\ref{d3}) with $\mathcal O_{\mathbb P^{3}}(k)$, and we get $h^{1}(\mathcal I_{C}(k)) = 0$ for any $k$.
Hence, $C$ is projectively normal. We can also get the Hilbert polynomial of $C$. Actually, when $k$ is large enough, we have
\begin{eqnarray*}
h^{0}(\mathcal I_{C}(k)) & = & (s+1)\cdotp h^{0}(\mathcal O_{\mathbb P^{3}}(k-s)) - s\cdotp h^{0}(\mathcal O_{\mathbb P^{3}}(k-s-1)) \\
                                        & = & \frac{1}{6}(k-s+2)(k-s+1)(k+2s+3).
\end{eqnarray*}

The Hilbert polynomial of $C$ equals
$$h^{0}(\mathcal O_{\mathbb P^{3}}(k))-h^{0}(\mathcal I_{C}(k)) = \frac{1}{2}(s^{2}+s)k - \frac{1}{6}(2s^{3}-3s^{2}-5s).$$

So immediately we obtain the degree and genus of $C$,
$$ d=\frac{1}{2}s(s+1), $$
$$ g = 1+\frac{1}{6}(2s^{3}-3s^{2}-5s). $$

If we take all possible linear forms as entries of $A$, by the
above construction, we get an irreducible component $U$ in the Hilbert scheme whose general point $[C]$ corresponds to a smooth irreducible and nondegenerate curve. By counting parameters, the dimension of $U$ is
$$ 4s(s+1)-1-\mbox{dim PGL}_{s}-\mbox{dim PGL}_{s+1} = 2s^{2}+2s = 4d. $$

Actually $U$ is smooth at $[C]$, due to the fact that $H^{1}(\mathcal N_{C})=0$.
By \cite[Remark 2.2.6]{Kl}, we have,
\begin{eqnarray}
\label{H0}
H^{0}(\mathcal N_{C}) = \mbox{Ext}^{1}(\mathcal I_{C}, \mathcal I_{C}), \nonumber \\
\label{H1}
H^{1}(\mathcal N_{C}) = \mbox{Ext}^{2}(\mathcal I_{C}, \mathcal I_{C}).
\end{eqnarray}
Apply the functor Hom$(-, \mathcal I_{C})$ to the exact sequence (\ref{d3}). We get a long exact sequence
\begin{eqnarray}
\label{LE}
0 \rightarrow \mbox{Hom}(\mathcal I_{C}, \mathcal I_{C}) \rightarrow \mbox{Hom}(\mathcal O^{\oplus (s+1)}_{\mathbb P^{3}}(-s), \mathcal I_{C})  \rightarrow   \mbox{Hom}(\mathcal O^{\oplus s}_{\mathbb P^{3}}(-s-1), \mathcal I_{C}) \nonumber \\
   \hookrightarrow \mbox{Ext}^{1}(\mathcal I_{C}, \mathcal I_{C}) \rightarrow \mbox{Ext}^{1}(\mathcal O^{\oplus (s+1)}_{\mathbb P^{3}}(-s), \mathcal I_{C})  \rightarrow  \mbox{Ext}^{1}(\mathcal O^{\oplus s}_{\mathbb P^{3}}(-s-1), \mathcal I_{C}) \nonumber \\
   \hookrightarrow \mbox{Ext}^{2}(\mathcal I_{C}, \mathcal I_{C}) \rightarrow  \mbox{Ext}^{2}(\mathcal O^{\oplus (s+1)}_{\mathbb P^{3}}(-s), \mathcal I_{C}) \rightarrow \cdots \ \ \ \ \ \ \ \ \ \ \ \ \ \ \ \ \ \ \ \ \ \ \ \ \ \ \
\end{eqnarray}
Note that $ \mbox{Ext}^{2}(\mathcal O_{\mathbb P^{3}}(-s), \mathcal I_{C}) = H^{2}(\mathcal I_{C}(s)). $
Twist (\ref{d3}) by $\mathcal O_{\mathbb P^{3}}(s)$, and we get $H^{2}(\mathcal I_{C}(s))=0$.
Moreover, $\mbox{Ext}^{1}(\mathcal O_{\mathbb P^{3}}(-s-1), \mathcal I_{C}) = H^{1}(\mathcal I_{C}(s+1)) = 0$, since $C$ is
projectively normal. Then from (\ref{H1}) and (\ref{LE}), it follows that $H^{1}(\mathcal N_{C}) = 0$ as we expect.

Let us look at the values of $d$ and $g$ obtained above. One observation is that
$g^{2}\sim \frac{8}{9} d^{3}$ asymptotically. It implies that the
ratio $\frac{g^{2}}{d^{3}}$ might be an important index. More precisely, we want to find a number
$\lambda$ such that if $\frac{g^{2}}{d^{3}}\leq \lambda$
asymptotically, then there exists a component of the Hilbert scheme whose dimension is close to $4d$. In this case, the lower bound $4d$ is still good. Actually, we will show that $\lambda = 1$ is almost the best.

Continue to consider determinantal curves. Modify the entries of the matrix $A$ by using degree $t$ homogeneous polynomials instead of linear forms.
Then the ideal sheaf of $C$ has resolution 
\begin{equation}
\label{IIE}
0\rightarrow \mathcal O^{\oplus s}_{\mathbb P^{3}}(-t-ts)\rightarrow \mathcal O^{\oplus (s+1)}_{\mathbb P^{3}}(-ts)\rightarrow \mathcal I_{C} \rightarrow 0.
\end{equation}
Compute the Hilbert polynomial as before. We obtain the degree and genus of $C$ as follow, 
$$ d = \frac{1}{2}s(s+1)t^{2}, $$
$$ g = 1 + \frac{1}{6}s(s+1)(2s+1)t^{3}-s(1+s)t^{2}. $$
By counting parameters, the dimension of the component of such curves is
\begin{eqnarray*}
4s(s+1)\binom{t+3}{3}-1-\mbox{dim PGL}_{s}-\mbox{dim PGL}_{s+1} \\
= \frac{1}{6}s(s+1)(t^{3}+6t^{2}+11t-6). 
\end{eqnarray*}
We denote the above value by $l$. 

Note that the ratio $\frac{g^{2}}{d^{3}}$ satisfies the inequality
$$ \frac{g^{2}}{d^{3}} < \frac{2}{9}(4+\frac{1}{s^{2}+s}). $$
So asymptotically $\frac{g^{2}}{d^{3}}\leq 1$. Moreover, the ratio tends to 1 if and only if $s=1$, that is, when $C$ is a complete intersection of two degree $t$ surfaces. 

Another interesting fact is that when $t\leq 3$ we always have $l=4d$ for any $s$. But as $t$ increases, $l$ will get much larger than $4d$.
We have already discussed the case $t=1$. Now we take a look at $t=2$ and $3$. 

If $t=2$, we have
$$ d = 2s(s+1), $$
$$ g = 1 + \frac{8}{3}(s-1)s(s+1). $$
Asymptotically $\frac{g^{2}}{d^{3}}$ goes to $\frac{8}{9}$. 

If $t=3$, we have
$$ d = \frac{9}{2}s(s+1), $$
$$ g = 1 + \frac{9}{2}s(s+1)(2s-1). $$
Asymptotically $\frac{g^{2}}{d^{3}}$ goes to $\frac{8}{9}$ as well, still less than 1. 

Next, we further modify the matrix $A$ by allowing the entries at different rows to have
different degrees. Suppose $A=(F_{ij})$, $1\leq i \leq s, 1\leq j \leq s+1$, and the degree of $F_{ij}$ is $k_{i}$.
Let $t=\sum_{i=1}^{s}k_{i}$. Then the ideal sheaf of $C$ has resolution
\begin{equation}
\label{IIIE}
0\rightarrow \bigoplus_{i=1}^{s}\mathcal O_{\mathbb P^{3}}(-t-k_{i})\rightarrow \mathcal O^{\oplus (s+1)}_{\mathbb P^{3}}(-t)\rightarrow \mathcal I_{C} \rightarrow 0.
\end{equation}
We can obtain the degree and genus of $C$,
$$ d = \frac{1}{2}\big(t^{2}+\sum_{i=1}^{s}k_{i}^{2}\big), $$
$$ g = 1 + \frac{1}{6}\Big(2t^{3}-6t^{2}+3\big(\sum_{i=1}^{s}k_{i}^{2}\big)t+\sum_{i=1}^{s}\big(k_{i}^{3}-6k_{i}^{2}\big)\Big). $$
In this case, the dimension estimate by counting parameters depends on how many $k_{i}$'s may have the same value.
However, we are more interested in if $\frac{g^{2}}{d^{3}}$ can approach a better upper bound, e.g., much larger than 1.

Let $u=\sum_{i=1}^{s}k_{i}^{2}$, then $\sum_{i=1}^{s}k_{i}^{3}\leq tu$. Hence, we have
$$ g < \frac{1}{6}(2t^{2}+4ut), $$
$$ \frac{g^{2}}{d^{3}} <  \frac{8}{9}\cdotp\frac{(1+2\alpha)^{2}}{(1+\alpha)^{3}}, $$
where $\alpha = \frac{u}{t^{2}}.$
When $\alpha = \frac{1}{2}$, the right hand side of the last inequality has the maximum $\frac{256}{243}$, which is still close to 1.

The above examples partially explains why we separate the case $g^{2} < d^{3}$ in Theorem \ref{r=3}. The reader may also wonder why we do not construct other examples in addition to determinantal curves to see if we can get
$\frac{g^{2}}{d^{3}}$ much larger than 1 and at the
same time keep the dimension of the component relatively low. In
fact, this is impossible. Next subsection explains the different
situation when $g^{2}\geq d^{3}$.

\subsection{Curves on a surface in $\mathbb P^{3}$}
In this subsection, we want to show if $g$ is much larger than $d$, then a curve must be contained in a relatively low degree surface in $\mathbb P^{3}$. Moreover, we can estimate the deformation of the curve on that surface, which provides a proof for the second part of Theorem \ref{r=3}.

Firstly, we cite a result originally mentioned by Halphen and proved later by Gruson and Peskine \cite{GP}.

\begin{theorem}
\label{s3}
Let $C$ be a connected smooth curve of degree $d$ and genus $g$ in $\mathbb P^{3}$. $s$ is a positive integer such that
$s(s-1)<d$. If $g$ satisfies
\begin{equation}
\label{g3}
g > \frac{d}{2}(s+\frac{d}{s}-4)-\frac{r(s-r)(s-1)}{2s},
\end{equation}
where $0\leq r < s, d+r\equiv 0$ (mod $s$),
then $C$ must lie on a surface of degree less than $s$.
\end{theorem}

Note that if $s\sim \sqrt{d}$, then the right hand side of (\ref{g3}) $\sim \sqrt{d^{3}}$. Hence, Theorem \ref{s3} can help us deal with the case $g^{2}>d^{3}$.

Since we only want asymptotic results, Theorem \ref{s3} can be slightly modified
for our convenience.

\begin{proposition}
\label{s3m}
Let $C$ be a connected smooth curve of degree $d$ and genus $g$ in $\mathbb P^{3}$. $s$ is a positive integer such that
$s(s+1)< d$. If $g$ satisfies
\begin{equation}
\label{g3m}
g > \frac{d}{2}(s+\frac{d}{s+1}-3),
\end{equation}
then $C$ must lie on a surface of degree $k \leq s$.
\end{proposition}

For fixed $d$ and $g$ in the range $g^{2}\geq d^{3}$, consider the smallest positive integer $s$ satisfying $s(s+1)< d$ and the inequality (\ref{g3m}). Then there exists a surface $S$ of degree $k\leq s$ such that $S$ contains $C$. Let $\mathcal H_{d,g}(S)$ be the Hilbert scheme parameterizing
degree $d$ and genus $g$ curves on $S$. $\mathcal H_{d,g}(S)$ can
be viewed as a subscheme of $\mathcal H_{d,g,3}$. We want to estimate $\mbox{dim}_{[C]}\mathcal H_{d,g}(S)$.

If $S$ is smooth, then $\chi(\mathcal N_{C/S})$ provides a lower bound for
$\mbox{dim}_{[C]}\mathcal H_{d,g}(S)$. We have the exact sequence
\begin{equation}
\label{nb}
 0\rightarrow \mathcal N_{C/S}\rightarrow \mathcal N_{C/\mathbb P^{3}}\rightarrow \mathcal N_{S/\mathbb P^{3}}\otimes \mathcal O_{C} \rightarrow 0.
\end{equation}
By adjunction formula, $\mathcal N_{S/\mathbb P^{3}}\otimes\mathcal O_{C} = \mathcal O_{C}(k)$.
Then we can compute $\chi(\mathcal N_{C/S})$ by the exact sequence (\ref{nb}) and Riemann-Roch,
\begin{eqnarray}
\mathcal X(\mathcal N_{C/S}) &=& \mathcal X(\mathcal N_{C/\mathbb P^{3}}) - \mathcal X(\mathcal N_{S/\mathbb P^{3}}\otimes \mathcal O_{C})\nonumber \\
                                                     &=& 4d - \mathcal X(\mathcal O_{C}(k)) \nonumber \\
                                                     &=& 4d+g-1-kd \nonumber \\
                                                     &\geq & 4d+g-1-sd. \nonumber
\end{eqnarray}
So we have
\begin{equation*}
\mbox{dim}_{[C]}\mathcal H_{d,g,3}\geq \mbox{dim}_{[C]}\mathcal H_{d,g}(S)\geq 4d+g-1-sd.
\end{equation*}
Therefore, we get a lower bound for the dimension of $\mathcal H_{d,g,3}$,
\begin{equation}
\label{se3}
l_{d,g,3}\geq 4d+g-1-sd.
\end{equation}

The advantage of (\ref{se3}) is because in the range $g^{2}\geq d^{3}$, as $g$ increases, $s$ decreases, and
 $4d+g-1-sd$ is more dominated by $g$. For instance, if we fix $d$ and let $g$ approach the Castelnuovo bound $\pi(d,3)$, then the dimension of $\mathcal H_{d,g,3}$ tends to $g$. But at this moment $s$ is very small. Therefore, the estimate (\ref{se3}) does not lose much information from the asymptotic viewpoint.

Now we can finish the proof of Theorem \ref{r=3} easily.

\begin{proof}
$4d$ is the classical lower bound for any $d,g$. Moreover, in the range $g^{2}\geq d^{3}$,
the smallest integer $s$ satisfying $s(s+1)< d$ and $g > \frac{d}{2}(s+\frac{d}{s+1}-3)$ is given
by $s = \mu (d,g)$. Apply the lower bound $4d+g-1-sd$ obtained in (\ref{se3}). It then completes the proof.
\end{proof}

In the above argument, there is one gap we need to fix, that is, when the surface $S$ is singular and $C$ passes through singular points of $S$. In that case we cannot simply apply cohomology to estimate the dimension of the deformation of
$C$ on $S$. Instead, we have to use Ext groups. Before doing that,
we will prove a simple result, which shows that the situation is not very bad even
if $S$ is singular.

\begin{lemma}
\label{fs}
Let $S_{sing}$ denote the singular locus of a surface $S$. Under the assumption of Proposition \ref{s3m}, if $C\cap S_{sing}$ is not empty, then it is 0-dimensional.
\end{lemma}

\begin{proof}
If the dimension of $S_{sing}$ is 0, then the statement is trivial. Otherwise the dimension of $S_{sing}$ is 1. By B\'{e}zout,
the degree of $S_{sing}$ is at most $k(k-1)\leq s(s-1)< d$. Hence, $C$ cannot be contained in $S_{sing}$.
\end{proof}

By Lemma \ref{fs}, we can apply the following result from \cite[Lemma 2.13, Theorem 2.15]{Ko}.

\begin{proposition}
\label{K3}
Keep the above notation. If $C\cap S_{sing}$ is 0-dimensional, then $C\subset S$ is generically unobstructed and the
dimension of every irreducible component of $\mathcal H_{d,g}(S)$ at $[C]$ is at least
\begin{eqnarray}
\label{ext3}
\mathrm{dim \ Hom}_{C}(\mathcal I_{C/S}/\mathcal I_{C/S}^{2}, \mathcal O_{C}) - \mathrm{dim \ Ext}_{C}^{1}(\mathcal I_{C/S}/\mathcal I_{C/S}^{2}, \mathcal O_{C}).
\end{eqnarray}
\end{proposition}

If $S$ is smooth, the value of (\ref{ext3}) is just $\mathcal X(\mathcal N_{C/S})$. When $S$ is singular, we need to verify some exact sequences of K\"{a}hler differentials. We will do it in a more general setting since the results can be applied to many other cases.

\begin{proposition}
\label{Kn}
Suppose $C$ is a smooth connected curve, $X$ is an $(n-k)$-dimensional local complete intersection, and $C\subset X\subset \mathbb P^{n}, \ n\geq 3, 1\leq k\leq n-2$. If $C\cap X_{sing}$ is 0-dimensional, we have the following exact sequences
\begin{eqnarray}
\label{cxn}
0\rightarrow\mathcal I_{C/X}/\mathcal I_{C/X}^{2}\xrightarrow{d}\Omega_{X}\otimes \mathcal O_{C}\rightarrow\Omega_{C}\rightarrow 0, \\
\label{cxpn}
0\rightarrow (\mathcal I_{X}/\mathcal I_{X}^{2})\otimes\mathcal O_{C}\xrightarrow{d} \Omega_{\mathbb P^{n}}\otimes\mathcal O_{C}\rightarrow \Omega_{X}\otimes \mathcal O_{C}\rightarrow  0.
\end{eqnarray}
\end{proposition}

Note that if $X$ is smooth, those results are well-known. When $X$ is singular, the above sequences are still exact except the left hand sides may
not be injective, cf. \cite[II 8]{Ha}.

\begin{proof}
It suffices to verify that the map to the middle term is always injective for each sequence. Since the question is local, we only need to work on a local affine chart $U$. Suppose $x_{1},\ldots,x_{n}$ are the local coordinates, and $f_{1},\ldots,f_{k}$ locally cut out $X$ in $U$. We have $\Omega_{X}(U) = \Omega_{\mathbb P^{n}}\otimes \mathcal O_{X}(U)/(df_{1},\ldots, df_{k})$.

Firstly, let us verify (\ref{cxn}). Pick an element $g\in \mathcal I_{C/X}(U)$. Suppose we have
$$dg=\sum_{j=1}^{n}\frac{\partial g}{\partial x_{j}}dx_{j}=0
\in \Omega_{X}\otimes \mathcal O_{C}(U).$$
There also exist $a_{1},\ldots,a_{k}\in \mathcal O_{C}(U)$ such that restricted on $C$,
$$\frac{\partial g}{\partial x_{j}}=\sum_{i=1}^{k}a_{i}\frac{\partial f_{i}}{\partial x_{j}}, \ 1\leq j\leq n. $$
It follows that $d(g - \sum_{i=1}^{k}a_{i}f_{i}) = 0$ on $C$. Since $C$ is smooth, the vanashing of $g - \sum_{i=1}^{k}a_{i}f_{i}$ and its differential
on $C$ tell us that $g - \sum_{i=1}^{k}a_{i}f_{i}\in \mathcal I_{C}^{2}(U)$,
which implies $g = g - \sum_{i=1}^{k}a_{i}f_{i} = 0$ as elements in $\mathcal I_{C/X}/\mathcal I_{C/X}^{2}(U)$.

Next, let us verify the exactness of (\ref{cxpn}). Take an element $h = \sum_{i=1}^{k}b_{i}f_{i}\in \mathcal I_{X}(U)$.
If $dh=0$ restricted on $C$, since $f_{1},\ldots,f_{k}$ vanash on $C$, we have $$\sum_{i=1}^{k}b_{i}\frac{\partial f_{i}}{\partial x_{j}}dx_{j} = 0, \ 1\leq j\leq n$$
on $C$. Note that $X_{sing}\cap U$ consists of those points where the matrix
$$\Big(\frac{\partial f_{i}}{\partial x_{j}}\Big)_{1\leq i\leq k, 1\leq j\leq n}$$
drops rank.
Since $C\cap X_{sing}$ consists of at most finitely many points, $b_{1},\ldots,b_{k}$ must vanash at a non empty open subset of $C\cap U$, which forces that
they vanash completely on $C\cap U$. Hence, $h\otimes 1 = \sum_{i=1}^{k} f_{i}\otimes b_{i} = 0 \in (\mathcal I_{X}/\mathcal I_{X}^{2})\otimes \mathcal O_{C}(U).$
\end{proof}

Now consider the deformation of $C$ on $X$. We have the following result.

\begin{proposition}
\label{cx}
Keep the above assumption. If $C\cap X_{sing}$ is 0-dimensional, the dimension of every component of $\mathcal H_{d,g}(X)$ at $[C]$ is at least $$\mathcal X(\mathcal N_{C/\mathbb P^{n}})-\mathcal X(\mathcal N_{X/\mathbb P^{n}}|_{C}).$$
Moreover, suppose $X$ is a complete intersection cut out by hypersurfaces $F_{1},\ldots, F_{k}$, deg $F_{i} = d_{i}, i=1,\ldots, k$.
The above lower bound can be written explicitly as  $$(n+1-\sum_{i=1}^{k}d_{i})d+(k-n+3)(g-1).$$
\end{proposition}

\begin{proof}
By the assumption, $C\subset X$ is generically unobstructed, so we can apply the result from \cite[Lemma 2.13, Theorem 2.15]{Kl}. The local dimension of any component of $\mathcal H_{d,g}(X)$ at $[C]$ is at least
\begin{eqnarray}
\label{extcx}
\mathrm{dim \ Hom}_{C}(\mathcal I_{C/X}/\mathcal I_{C/X}^{2}, \mathcal O_{C}) - \mathrm{dim \ Ext}_{C}^{1}(\mathcal I_{C/X}/\mathcal I_{C/X}^{2}, \mathcal O_{C})
\end{eqnarray}
Note that if $X$ is smooth, the value of (\ref{extcx}) is $\mathcal X(\mathcal N_{C/X})$, which equals
$\mathcal X(\mathcal N_{C/\mathbb P^{n}})-\mathcal X(\mathcal N_{X/\mathbb P^{n}}|_{C})$ due to the well-known exact sequence
$$ 0 \rightarrow \mathcal N_{C/X}\rightarrow \mathcal N_{C/\mathbb P^{n}}\rightarrow \mathcal N_{X/\mathbb P^{n}}|_{C}\rightarrow 0.$$
If $X$ is singular, apply the functor Hom($\cdotp$, $\mathcal O_{C}$) to (\ref{cxn}). Then we get a long exact sequence
\begin{eqnarray}
\label{LExtn}
0 \rightarrow \mathrm{Hom} (\Omega_{C}, \mathcal O_{C}) \rightarrow \mathrm{Hom} (\Omega_{X}\otimes \mathcal O_{C}, \mathcal O_{C})\rightarrow \mathrm{Hom} (\mathcal I_{C/X}/\mathcal I_{C/X}^{2},\mathcal O_{C}) \nonumber \\
   \hookrightarrow \mathrm{Ext}^{1} (\Omega_{C}, \mathcal O_{C}) \rightarrow \mathrm{Ext}^{1} (\Omega_{X}\otimes \mathcal O_{C}, \mathcal O_{C})\rightarrow \mathrm{Ext}^{1} (\mathcal I_{C/X}/\mathcal I_{C/X}^{2},\mathcal O_{C}) \nonumber \\
   \hookrightarrow  0. \hspace{9.6cm}
\end{eqnarray}
The last term is zero, because $\mathrm{Ext}^{2}(\Omega_{C},\mathcal O_{C}) = H^{2}(\mathcal T_{C}) = 0$.

Moreover, apply the functor Hom($\cdotp$, $\mathcal O_{C}$) to (\ref{cxpn}), we get another long exact sequence
\begin{eqnarray}
\label{LLExtn}
0 \rightarrow \mathrm{Hom}(\Omega_{X}\otimes\mathcal O_{C},\mathcal O_{C})\rightarrow \mathrm{Hom}(\Omega_{\mathbb P^{n}}\otimes \mathcal O_{C},\mathcal O_{C})\rightarrow \mathrm{Hom}((\mathcal I_{X}/\mathcal I_{X}^{2})\otimes \mathcal O_{C},\mathcal O_{C})  \nonumber \\
  \hookrightarrow  \mathrm{Ext}^{1}(\Omega_{X}\otimes\mathcal O_{C},\mathcal O_{C})\rightarrow \mathrm{Ext}^{1}(\Omega_{\mathbb P^{n}}\otimes \mathcal O_{C},\mathcal O_{C})\rightarrow \mathrm{Ext}^{1}((\mathcal I_{X}/\mathcal I_{X}^{2})\otimes \mathcal O_{C},\mathcal O_{C}) \nonumber \\
\hookrightarrow \mathrm{Ext}^{2}(\Omega_{X}\otimes \mathcal O_{C},\mathcal O_{C})\rightarrow 0. \hspace{6.8cm}
\end{eqnarray}
The last term is zero, because $\mathrm{Ext}^{2}(\Omega_{\mathbb P^{n}}\otimes O_{C}, \mathcal O_{C}) = H^{2}(\mathcal T_{\mathbb P^{n}}|_{C})=0$.

Note that $C$ is smooth, so $\mathrm{Ext}^{i}(\Omega_{C}, \mathcal O_{C})=H^{i}(\mathcal T_{C})$
and $\mathrm{Ext}^{i}(\Omega_{\mathbb P^{n}}\otimes \mathcal O_{C},\mathcal O_{C}) = H^{i}(\mathcal T_{\mathbb P^{n}}|_{C})$ for any $i$.
From (\ref{cxpn}), we know $(\mathcal I_{X}/\mathcal I_{X}^{2})\otimes \mathcal O_{C}$ is locally free,
so $\mathrm{Ext}^{i}((\mathcal I_{X}/\mathcal I_{X}^{2})\otimes\mathcal O_{C}, \mathcal O_{C}) = H^{i}(\mathcal N_{X/\mathbb P^{n}}|_{C})$.
Then by (\ref{LExtn}) and (\ref{LLExtn}), we have
\begin{eqnarray}
\label{XCEn}
&&\mathrm{dim \ Hom}(\mathcal I_{C/X}/\mathcal I_{C/X}^{2}, \mathcal O_{C}) - \mathrm{dim \ Ext}^{1}(\mathcal I_{C/X}/\mathcal I_{C/X}^{2}, \mathcal O_{C})  \nonumber \\
&=& \mathcal X(\mathcal T_{\mathbb P^{n}}|_{C}) - \mathcal X(\mathcal N_{X/\mathbb P^{n}}|_{C}) - \mathcal X(\mathcal T_{C}) - \mathrm{dim \ Ext}^{2}(\Omega_{X}\otimes \mathcal O_{C},\mathcal O_{C})\nonumber \\
&=& \mathcal X(\mathcal N_{C/\mathbb P^{n}}) - \mathcal X(\mathcal N_{X/\mathbb P^{n}}|_{C}) - \mathrm{dim \ Ext}^{2}(\Omega_{X}\otimes \mathcal O_{C},\mathcal O_{C}). \nonumber
\end{eqnarray}
$\mathcal X(\mathcal N_{C/\mathbb P^{n}})$ equals $h_{d,g,n} = (n+1)d-(n-3)(g-1)$.
If $X$ is a complete intersection cut out by $F_{1},\ldots, F_{k}$, the normal sheaf $\mathcal N_{X/\mathbb P^{n}}$ splits into
$\bigoplus_{i=1}^{k}\mathcal O_{X}(d_{i})$. Therefore, in this case we can compute $\mathcal X(\mathcal N_{X/\mathbb P^{n}}|_{C})$ explicitly as
$\mathcal X(\bigoplus_{i=1}^{k}\mathcal O_{C}(d_{i})) = \sum_{i=1}^{k}(1-g+dd_{i})$.

Now the theorem follows if we can show that $\mathrm{Ext}^{2}(\Omega_{X}\otimes \mathcal O_{C},\mathcal O_{C}) = 0$.
In case $X$ is smooth, we have the well-known exact sequence
$$ 0\rightarrow \mathcal T_{X}\otimes \mathcal O_{C}\rightarrow \mathcal T_{\mathbb P^{n}}\otimes \mathcal O_{C}\rightarrow \mathcal N_{X/\mathbb P^{n}}\otimes\mathcal O_{C} \rightarrow 0.$$
If $X$ is singular, the last map may not be surjective. Instead, we have
$$ 0\rightarrow \mathcal T_{X}\otimes \mathcal O_{C}\rightarrow \mathcal T_{\mathbb P^{n}}\otimes \mathcal O_{C}\rightarrow \mathcal N_{X/\mathbb P^{n}}\otimes\mathcal O_{C} \rightarrow \mathcal F\rightarrow 0, $$
where $\mathcal F$ is a sheaf supported at some points of $C\cap X_{sing}$.
Split the above sequence into two short exact sequences
\begin{eqnarray}
\label{En}
&0\rightarrow \mathcal T_{X}\otimes \mathcal O_{C}\rightarrow \mathcal T_{\mathbb P^{n}}\otimes \mathcal O_{C}\rightarrow \mathcal E\rightarrow 0 &\\
\label{Fn}
&0\rightarrow \mathcal E \rightarrow \mathcal N_{X/\mathbb P^{n}}\otimes\mathcal O_{C} \rightarrow \mathcal F\rightarrow 0. &
\end{eqnarray}
Since $H^{2}(\mathcal T_{X}\otimes \mathcal O_{C}) = 0$, then from (\ref{En}), the map
$H^{1}(\mathcal T_{\mathbb P^{n}}\otimes \mathcal O_{C})\rightarrow H^{1}(\mathcal E)$ is surjective.
Moreover, $\mathcal F$ is only supported at finitely many points on $C$, so $H^{1}(\mathcal F) = 0$. From (\ref{Fn}),
the map $H^{1}(\mathcal E)\rightarrow H^{1}(\mathcal N_{X/\mathbb P^{n}}\otimes\mathcal O_{C})$ is also surjective.
Hence, we get a surjective map $H^{1}(\mathcal T_{\mathbb P^{n}}\otimes O_{C})\rightarrow H^{1}(\mathcal N_{X/\mathbb P^{n}}\otimes\mathcal O_{C}) $,
i.e. a surjective map $\mathrm{Ext}^{1}(\Omega_{\mathbb P^{n}}\otimes\mathcal O_{C},\mathcal O_{C})\rightarrow \mathrm{Ext}^{1}((\mathcal I_{X}/\mathcal I_{X}^{2})\otimes \mathcal O_{C},\mathcal O_{C})$.
Then from (\ref{LLExtn}), it follows that $\mathrm{Ext}^{2}(\Omega_{X}\otimes O_{C},\mathcal O_{C})=0$.
\end{proof}

Now, apply Proposition \ref{K3} and \ref{cx} to our situation when $X = S$ is a surface in $\mathbb P^{3}$. The bound $4d+g-1-sd$ is still valid as a lower bound for $l_{d,g,3}$. Now we have completely finished the proof of Theorem \ref{r=3}.

At the end of this section, we want to show that the new bound in Theorem \ref{r=3}
makes sense. Using determinantal curves, we already constructed components with the expected dimension $4d$ and the corresponding values of $g$ and $d$ satisfy
$\frac{g^{2}}{d^{3}}\sim \frac{8}{9}$. Actually, if $d$ is large and
$g\leq \frac{1}{6\sqrt{2}}d\sqrt{d}+k_{1}d+k_{2}\sqrt{d}+k_{3}$, where
$k_{1}, k_{2}, k_{3}$ are some constants, there always exists a component of
$\mathcal H_{d,g,3}$ with the expected dimension $4d$, cf. \cite{P}.
Moreover, we can construct those components
up to $g\sim \frac{2\sqrt{2}}{3}d\sqrt{d}$ asymptotically, cf. \cite{F} and \cite{W}. Therefore, in the range $g^{2}< d^{3}$, $4d$ is almost the best lower bound for $l_{d,g,3}$. On the other hand, in the range $g^{2}\geq d^{3}$
we always have $\mu(d,g) < \sqrt{d}$. Furthermore, as $g$
increases, $\mu(d,g)$ decreases and $4d+g-1-\mu(d,g)d$ is
dominated by $g$. In fact, we know the dimension of a component of 
$\mathcal H_{d,g,3}$ whose general points correspond to smooth irreducible and nondegenerate curves is always less than or equal to $4d+g$ for any $d,g$, cf. \cite[2.b]{H}. Therefore, the result of Theorem \ref{r=3}
does not lose much information from the asymptotic perspective. More importantly, it is better than the expected dimension $4d$ if $g$ is much larger than $d$.

\section{The Hilbert scheme of curves in $\mathbb P^{4}$}
In this section we will prove Theorem \ref{r=4}. The idea of the proof is simple. We will show that if $g$ is large enough, a degree $d$ genus $g$ smooth irreducible and nondegenerate curve $C$ in $\mathbb P^{4}$ must be contained in a surface $S$ such that $S$ is a complete intersection and $C$ is not contained in its singular locus $S_{sing}$. By estimating the dimension of the deformation of $C$ on $S$, we can derive the desired result.

For the first step, let us recall some basic results from the Castelnuovo theory.
\begin{theorem}
\label{Ca}
Let $C$ be a degree $d$ genus $g$ reduced irreducible and nondegenerate curve in $\mathbb P^{r}$. Then
$g$ has an upper bound $\pi(d,r) = \frac{d^{2}}{2(r-1)} + O(d)$.
\end{theorem}

For the precise definition of $\pi(d,r)$ and the proof of the theorem, cf. e.g., \cite{H}.

By the above theorem, it is easy to find a low degree threefold $F$ that contains $C$.
\begin{lemma}
\label{F}
Let $k$ be a positive integer and $N = {k+4\choose 4}-1$. If $g$ satisfies
\begin{equation}
\label{gk}
g > \pi(dk, N),
\end{equation}
then $C$ is contained in an irreducible threefold $F$ of degree $a\leq k$.
\end{lemma}

\begin{proof}
Embed $\mathbb P^{4}$ into $\mathbb P^{N}$ by the Veronese map of degree $k$ . Then the image $C'$ of $C$ is a curve of degree $dk$ and genus $g$.
Since $g$ is larger than the Castelnuovo bound $\pi(dk, N)$, $C'$ must be contained in a hyperplane in $\mathbb P^{N}$. That is, $C$ is contained in a degree $k$ threefold in
$\mathbb P^{4}$. Then we take an irreducible component $F$ of this threefold that contains $C$. $F$ has degree $a \leq k$.
\end{proof}

Fix $F$ and its degree $a$. Our next goal is to find another threefold that contains $C$ as well.

\begin{lemma}
\label{G}
Suppose $l$ is an integer and $l\geq a$. Let $M = {l+4\choose 4}-{l-a+4\choose 4}-1$. If $g$ satisfies 
\begin{equation}
\label{gl}
g> \pi(dl, M),
\end{equation}
then we can find a degree $b$ irreducible threefold $G$ containing $C$ such that $b\leq l$ and the surface $S=F\cap G$ is a complete intersection.
\end{lemma}

\begin{proof}
Embed $\mathbb P^{4}$ into $\mathbb P^{N}$ by the Veronese map of degree $l$. 
By a similar argument as before, we can show that $C$ is contained in at least ${l-a+4\choose 4}+1$ independent degree $l$ threefolds in $\mathbb P^{4}$. 
Notice that there are at most ${l-a+4\choose 4}$ independent degree $l$ threefolds containing $F$ as a component, since $F$ is irreducible. Hence, we can find a degree $l$ threefold containing $C$ but not $F$. Take an irreducible component $G$ of this threefold that contains $C$. $G$ has degree $b\leq l$ and 
$S=F\cap G$ is a complete intersection. 
\end{proof} 

In order to apply standard deformation theory for $C\subset S$, we should avoid the situation $C\subset S_{sing}$. 

\begin{lemma}
\label{Sing}
Let $S$ be a surface in $\mathbb P^{4}$ cut out by two threefolds of degree $a$ and $b$ respectively. If $S_{sing}$ is 1-dimensional, its degree has an upper bound $\frac{1}{2}ab(a+b-2)$. 
\end{lemma}

\begin{proof}
Take a general hyperplane section $X = H\cap S$ in $\mathbb P^{4}$. $X$ is a curve of degree $ab$ and arithmetic genus $\frac{1}{2}ab(a+b-4)+1$ in $H\cong \mathbb P^{3}$. Even though $X$ might be reducible, the total number of its singularities is at most $ab + \frac{1}{2}ab(a+b-4) +1 - 1 = \frac{1}{2}ab(a+b-2)$. 
Since $H\cap S_{sing}\subset X_{sing}$, we get deg $S_{sing} \leq$ deg $X_{sing} \leq \frac{1}{2}ab(a+b-2)$.    
\end{proof}

By this lemma, we immediately get the following consequence.
\begin{lemma}
\label{S}
Keep the above assumption. If the degree $d$ of the curve $C$ satisfies 
\begin{equation}
\label{dab}
d> \frac{1}{2}ab(a+b-2),
\end{equation} 
then $C\cap S_{sing}$ is either empty or 0-dimensional. 
\end{lemma}

Consider the deformation of $C$ on $S$. Since $S$ is a complete intersection and $C\not\subset S_{sing}$, we can apply Proposition \ref{cx} to derive 
the following result.
\begin{lemma}
\label{Def}
The dimension of the deformation of $C$ on $S$ is at least $5d+g-1-(a+b)d$. 
\end{lemma}

Now we have all the ingredients to prove Theorem \ref{r=4}. 
\begin{proof}
By an elementary calculation, if $g>3d\sqrt{d}+O(d)$, we can find integers $k, a, l, b$ successively in the above setting such that they satisfy the inequalities (\ref{gk}), (\ref{gl}) and (\ref{dab}). Therefore, by Lemma \ref{F}, \ref{G} and \ref{S}, we know that $C$ lies in a complete intersection surface $S$ of type $(a, b)$ and $C\not\subset S_{sing}$. Moreover, we can check that $(a+b)d < g$. Then by Lemma \ref{Def}, the dimension of the deformation of $C$ on $S$ $\geq 5d + g -1 - (a+b)d \geq 5d > 24 = $ dim PGL(5).    
\end{proof}

It is possible to enlarge the range $g>3d\sqrt{d}+O(d)$ by refining the results in Lemma \ref{F}, \ref{G} and \ref{Sing}. However, it seems that only the leading coefficient could be improved rather than the exponent $d^{3/2}$. So when $g$ is slightly bigger than $d$, the situation remains mysterious to us. On the other hand, by the result of \cite{CCG}, Conjecture \ref{r>4} mentioned in the introduction section sounds highly possible and might be handled by an analogous argument. We state the conjecture again as the end of this section.

\begin{conjecture}
For $r\geq 5$, there always exists a constant $\lambda_{r}$ such that if $g\geq \lambda_{r}d\sqrt{d}+O(d)$, a degree $d$ genus $g$ smooth irreducible 
and nondegenerate curve in $\mathbb P^{r}$ is not rigid. 
\end{conjecture}

\section{The Hilbert scheme of curves on a quadric threefold}
In this section we will prove Theorem \ref{quadric}. Recall that 
$\mathcal H_{d, g}(Q)$ parameterizes degree $d$ genus $g$ smooth irreducible and nondegenerate curves on a smooth quadric $Q$ in $\mathbb P^{4}$. 
For [$C$]$\in \mathcal H_{d,g}(Q)$,  
$\mathcal X(\mathcal N_{C/Q}) = 3d$ is a lower bound for the dimension of any component of $\mathcal H_{d,g}(Q)$. Theorem \ref{quadric} provides a further 
analysis for the sharpness of this bound. Its proof consists of two steps. 

Firstly, if $g$ is large enough, $C$ must lie on another threefold $F$ of low degree. Consider the deformation of $C$ on the surface $X = Q\cap F$. We can easily derive the first part of Theorem \ref{quadric}. 
For the second part, we use a similar method as in \cite{P}. A component whose general element represents a curve as the intersection of $Q$ and a determinantal surface has dimension $3d$. Then we apply the 
smoothing technique in \cite{S} to enlarge the range of the pair ($d, g$) to cover the case when $g<\frac{2}{15\sqrt{5}} d\sqrt{d} + O(d)$.

By the main result of \cite{C}, we can verify the first step easily. 

\begin{lemma} 
\label{GB}
If $g> \frac{1}{\sqrt{2}}d\sqrt{d} + O(d)$, the dimension of the deformation of $C$ on $Q$ is bigger than $3d$. 
\end{lemma}

 \begin{proof}
When $d$ and $g$ satisfy the above inequality, we can find an integer $k$ such that $d>2k(k-1)$ and $g>\frac{d^{2}}{4k}+\frac{1}{2}kd$. 
By the result of \cite{C}, there exists an intergral surface $X\in |\mathcal O_{Q}(a)|$ containing $C$, where $a\leq k$. 
Since $d>2k(k-1)$ and $X$ is of degree $2a$, $C\not\subset X_{sing}$. By Proposition \ref{cx},  
$\mathcal X(\mathcal N_{C/X}) = 3d + g - ad - 1$ provides a lower bound for the dimension of the deformation of $C$ on $X$.     
A simple calculation shows that $3d + g - ad - 1 \geq 3d + g - kd - 1 > 3d$. 
\end{proof}

The second step is harder. We still want to construct a component of the Hilbert scheme that parameterizes certain determinantal curves. But the curves should be contained in the quadric $Q$. A natural idea is to take the intersection of a determinantal surface with $Q$. 

Let $\big(H_{ij}\big)$ be a $t\times (t+1)$ matrix. The entry $H_{ij}$ is a general linear form in $\mathbb P^{4}$. Those $t\times t$ minors 
define a determinantal surface $S$. The ideal sheaf of $S$ has the following resolution
$$ 0 \rightarrow \mathcal O^{\oplus t}_{\mathbb P^{4}}(-t-1)\rightarrow \mathcal O^{\oplus (t+1)}_{\mathbb P^{4}}(-t) \rightarrow \mathcal I_{S}\rightarrow 0.$$
By Bertini, if we take a general quadric threefold $Q$, $C=Q\cap S$ is smooth. It is not hard to get the degree and genus of $C$,
$$ d = t(t+1), $$ 
$$ g = \frac{2}{3}t^{3} - \frac{1}{2}t^{2} - \frac{7}{6}t + 1.$$
Note that asymptotically $g \sim \frac{2}{3}d\sqrt{d}$.

Let us count parameters. The dimension of the component parameterizing curves generated in the above way is
$5t(t+1) - 1 - \text{dim PGL}(t) - \text{dim PGL}(t+1) = 3t(t+1) = 3d$. In order to show that this is a real
component of $\mathcal H_{d,g}(Q)$, we have to check that for $C = S\cap Q$, $H^{1}(\mathcal N_{C/Q}) = 0$. 
Actually for general $S$ and $Q$, the ideal sheaf $\mathcal I_{C/Q}$ has the resolution 
$$ 0\rightarrow \mathcal O^{\oplus t}_{Q}(-t-1) \rightarrow \mathcal O^{\oplus (t+1)}               _{Q}(-t)\rightarrow \mathcal I_{C/Q}\rightarrow 0. $$ 
By \cite[Remark 2.2.6]{Kl}, we know that $H^{1}(\mathcal N_{C/Q}) = \text{Ext}_{Q}^{2}(\mathcal I_{C/Q}, \mathcal I_{C/Q})$.
Apply the functor $\text{Hom}_{Q}(-, \mathcal I_{C/Q})$ to the exact sequence. Then it is easy to derive the conclusion $H^{1}(\mathcal N_{C/Q})= 0$. 

The above construction is nice. But it has strong restriction on the values of $d$ and $g$. We really want to extend the result to more general values of $d$ and $g$. Here we will follow the methods in \cite{P} and \cite{S}. The idea works as follows. 
Take a smooth determinantal curve $\Gamma$ constructed as above and a smooth rational curve $\gamma$ on $Q$ such that they meet transversely. 
Further assume that 
$H^{1}(\mathcal N_{\Gamma/Q})=H^{1}(\mathcal N_{\gamma/Q})=0$. Then the nodal curve $\Gamma\cup \gamma$ can be smoothed out in $Q$. 
Moreover, the vanishing property of $H^{1}(\mathcal N)$ is locally preserved under this smoothing process. Then after smoothing the nodal curve, we may get the degree and genus in a more general range. 

Firstly, let us introduce an important smoothing technique used in \cite{S}.  

\begin{lemma}
\label{SM}
Let $\Gamma' = \Gamma \cup \gamma$ be a nodal union of two smooth irreducible curves on the quadric threefold $Q$. 
$\Gamma \cap \gamma = {P_{1},\ldots, P_{\delta}}$.  If 
$H^{1}(\mathcal N_{\Gamma/Q})=H^{1}(\mathcal N_{\gamma/Q}) = H^{1}(\mathcal N_{\gamma/Q}(-P_{1}-\ldots-P_{\delta}))=0$, then 
$H^{1}(\mathcal N_{\Gamma'/Q})=0 $ and $\Gamma'$ is smoothable in $Q$. 
\end{lemma}

 \begin{proof}
Let us first set up some notation. For a connected reduced curve $C$ on $Q$, denote $\mathcal N'_{C/Q}$ as the cokernel   
 of the map $\mathcal T_{C}\rightarrow \mathcal T_{Q|C}$ and let $\mathcal T^{1}_{C/Q}$ be the cotangent sheaf of $C$ in $Q$. 
 $\mathcal T^{1}_{C/Q}$ can be defined as the cokernel of the map $\mathcal N'_{C/Q}\rightarrow \mathcal N_{C/Q}$. 
 Suppose the singularities of $C$ are only nodes. Then $\mathcal T^{1}_{C/Q}$ is a torsion sheaf supported on each node of $C$. 
 Furthermore, if $H^{1}(\mathcal N'_{C/Q}) = 0$, by the argument of \cite[Proposition 1.6]{S}, $C$ is smoothable in $Q$. 
 
 Now in our case, the ideal sheaves $\mathcal I_{\Gamma/\Gamma'}\cong \mathcal O_{\gamma}(-P_{1}-\ldots - P_{\delta})$  
 and $\mathcal I_{\gamma/\Gamma'}\cong \mathcal O_{\Gamma}(-P_{1}-\ldots-P_{\delta})$. As in \cite[Lemma 5.1]{S}, we can also
 establish two exact sequences of sheaves on $\Gamma'$,
 $$ 0\rightarrow \mathcal I_{\Gamma/\Gamma'}\otimes \mathcal N_{\Gamma'/Q} \rightarrow \mathcal N'_{\Gamma'/Q}\rightarrow \mathcal N_{\Gamma/Q}\rightarrow 0, $$
 $$ 0\rightarrow \mathcal N_{\gamma/Q}(-P_{1}-\ldots-P_{\delta})\rightarrow \mathcal I_{\Gamma/\Gamma'}\otimes \mathcal N_{\Gamma'/Q}\rightarrow \mathcal T^{1}_{\Gamma'/Q}\rightarrow 0. $$
By the assumption and the fact that $H^{1}(\mathcal T^{1}_{\Gamma'/Q}) = 0$, we get $H^{1}(\mathcal N'_{\Gamma'/Q})=0$ from the long exact sequences of cohomology. 
Hence, $\Gamma'$ is smoothable in $Q$. Moreover, the map $\mathcal N'_{\Gamma'/Q}\rightarrow \mathcal N_{\Gamma'/Q}$ is injective 
and its cokernel $T^{1}_{\Gamma'/Q}$ is supported at the nodes. So $H^{1}(\mathcal N'_{\Gamma'/Q})=0$ implies that $H^{1}(\mathcal N_{\Gamma'/Q})=0$. 
\end{proof}

We still need another source curve $\gamma$. Here we will consider rational normal curves in $\mathbb P^{4}$ that lie on the quadric $Q$. 

\begin{lemma}
\label{RC}
Let $R=Q \cap H$ be a general hyperplane section of $Q$. $P_{1},\ldots, P_{m}\in R$ are $m\geq 6$ points in general position. For any integer $\delta\leq 4$, 
there exists a rational normal curve $\gamma \subset R$ such that $\gamma$ passes through exactly $\delta$ points of $P_{1},\ldots, P_{m}$. Furthermore, suppose those points on $\gamma$ are 
$P_{1},\ldots, P_{\delta}$. Then we have $H^{1}(\mathcal N_{\gamma/Q}(-P_{1}-\ldots - P_{\delta}))=0$.
\end{lemma}

\begin{proof}       
$R$ is a smooth quadric surface in $H\cong \mathbb P^{3}$. It is easy to find a degree 3 smooth rational curve $\gamma$ on $R$ that passes through 
$\delta$ general points, say, $P_{1},\ldots, P_{\delta}$. By the exact sequence 
$$ 0 \rightarrow \mathcal N_{\gamma/R} \rightarrow \mathcal N_{\gamma/Q} \rightarrow \mathcal N_{R/Q}\otimes \mathcal O_{\gamma} \rightarrow 0, $$ we can also get  
$H^{1}(\mathcal N_{\gamma/Q}(-P_{1}-\ldots - P_{\delta}))=0$.
\end{proof} 

Now we have all the ingredients to prove the second part of Theorem \ref{quadric}. 
\begin{proof}
Take a determinantal curve $\Gamma\subset Q$ whose degree $d_{\Gamma} = t(t+1)$ and genus $g_{\Gamma} = \frac{2}{3}t^{3}-\frac{1}{2}t^{2}-\frac{7}{6}t+1$.    
Consider a general hyperplane section of $\Gamma$. We get $d_{\Gamma}$ points in general position. By Lemma \ref{SM} and \ref{RC}, we can pick a suitable 
degree $3$ rational curve $\gamma$, such that $\Gamma\cap\gamma$ consists of $\delta$ reduced points for any $\delta\leq 4$ and the nodal union $\Gamma \cup\gamma$ is 
smoothable. Hence, starting from the pair $(d_{\Gamma}, g_{\Gamma})$, we can get a new pair $(d' = d_{\Gamma} + 3, g' = g_{\Gamma}+\delta - 1)$ where the Hilbert scheme 
$\mathcal H_{d',g'}(Q)$ also has a component of expected dimension $3d'$. Do the same step again and eventually it covers every pair $(d, g)$ in the form $(d_{\Gamma} + 3k, g_{\Gamma} + h)$, $0\leq h\leq 3k$.

Now we fix $d$. Note that $d_{\Gamma}=t(t+1)\equiv 0$ or 2 (mod 3). So if $d \equiv 0$ (mod 3), by the above construction, the range of $g$ for which 
$\mathcal H_{d,g}(Q)$ has a component of dimension $3d$ contains the following,
$$\frac{1}{6}(4t^{3}-3t^{2}-7t+6)\leq g \leq \frac{1}{6}(4t^{3}-3t^{2}-7t+6) + 3\cdot\frac{d -t(t+1)}{3}, $$      
for any $t(t+1) \leq d$ and $t\equiv 0$ or 2 (mod 3). 
In order to cover the case $d \equiv 1$ or 2 (mod 3), we can use a suitable line $l$ on $Q$ instead of the rational curve $\gamma$ in Lemma \ref{RC} such that $l$
intersects the source curve only at one point $P$. One can easily check that $H^{1}(\mathcal N_{l/Q}) = H^{1}(\mathcal N_{l/Q}(-P)) = 0$ hold. Then after smoothing the nodal union of $l$ and the source curve, this construction provides $(d-1,g)\rightarrow (d,g)$. So if $d\not\equiv 0$ (mod 3),
we can always consider $d-1$ or $d-2$ instead. 
In sum, the desired range of genus includes  
$$ L(t) =\frac{1}{6}(4t^{3}-3t^{2}-7t+6)\leq g \leq \frac{1}{6}(4t^{3}-3t^{2}-7t+6) + d -t(t+1) - 2 = R(t), $$   
where $t(t+1)\equiv 0$ (mod 3). 
Since $t\equiv 0$ or 2 (mod 3), each time $t$ increases by 1 or 2. In order to avoid that the value of $g$ jumps for a fixed $d$, 
we have to require that $L(t+2)\leq R(t)$. Solving this inequality and plugging the upper bound of $t$ into $R(t)$, 
we get the desired range of $g$ up to $\frac{2}{15\sqrt{5}} d\sqrt{d} + O(d)$.      
\end{proof}   

\begin{remark}
We can obtain a similar result for the Hilbert scheme of curves on a general cubic threefold Y. It is easy to check that the curve $C$ cut out by a determinantal surface and $Y$ also satisfies $H^{1}(\mathcal N_{C/Y})=0$. However, when we resume the process to quartic threefolds, the determinantal 
model does not work any longer. Another long standing problem is about quintic threefolds, since the expected dimension of the Hilbert scheme is 0 in that case. Even for rational curves, the famous Clemens' conjecture has been only solved when the degree of the curve is small. If we consider threefolds of higher degree, things become further unclear. To the author's best knowledge, 
we even do not know if a general threefold of degree $k>5$ in $\mathbb P^{4}$ contains an irreducible curve whose degree is not divisible by $k$. In sum, the Hilbert scheme of curves on a threefold of higher degree remains mysterious to us.
\end{remark}

Department of Mathematics, Statistics and Computer Science,
University of Illinois at Chicago, 851 S Morgan St, Chicago, IL 60607 \par
{\it Email address:} dwchen@math.uic.edu


\begin{thebibliography}{[C9]}

\bibitem{C}
M. A. A. De Cataldo, \textit{The genus of curves on the three-dimensional quadric}, Nagoya Math. J. 147, 193-211, 1997

\bibitem{CCG}
L. Chiantini, C. Ciliberto, and V. Di Gennaro, \textit{The genus of projective curves}, Duke Math. J. 70, no. 2, 229-245, 1993

\bibitem{E}
P. Ellia, \textit{Exemple de courdes de $\mathbb P^{3}$ \`{a} fibr\'{e} normal semi-stable, stable}, Math. Ann. 264, 389-396, 1983

\bibitem{F}
G. Fl\o ystad, \textit{Construction of space curves with good properties},
Math. Ann. 289, 33-54, 1991

\bibitem{GP}
L. Gruson, C. Peskine, \textit{Genre des courbes de l'espace projectif},
Algebraic geometry (Proc. Sympos., Univ. Troms\o , Troms\o , 1977), 31-59,
Lecture Notes in Math. 687, Springer, Berlin, 1978

\bibitem{H}
J. Harris, \textit{Curves in projective spaces}, Montr\'{e}al, Les Presses de L'Universit\'{e} de Montreal, 1982

\bibitem{HM}
J. Harris, I. Morrison, \textit{Moduli of curves}, Springer, 1998

\bibitem{Ha}
R. Hartshorne, \textit{Algebraic geometry}, Spring-Verlag, 1977

\bibitem{Kl}
J. O. Kleppe, \textit{The Hilbert-flag scheme, its properties and its connection with the Hilbert scheme}, PhD Thesis, University of Oslo, 1981

\bibitem{Ko}
J. Koll\'{a}r, \textit{Rational curves on algebraic varieties}, Springer-Verlag, 1996

\bibitem{P}
G. Pareschi, \textit{Components of the Hilbert scheme of smooth space curves with the expected number of moduli},
Manuscripta Math. 63, 1-16, 1989

\bibitem{S}
E. Sernesi, \textit{On the existence of certain families of curves}, 
Invent. Math. 75, 25-57, 1984

\bibitem{W}
C. Walter, \textit{Curves in $\mathbb P^{3}$ with the expected monad},
J. Algebraic Geom. 4, 301-320, 1995

\end{thebibliography}
\end{document}